\documentclass[11pt]{amsart}

\usepackage{amsmath}
\usepackage{dsfont}

\newtheorem{theorem}{Theorem}[section]

\newtheorem{lemma}[theorem]{Lemma}
\newtheorem{corollary}[theorem]{Corollary}

\def\d{\delta}
\def\e{\varepsilon}
\def\l{\lambda}
\def\m{{\mu}}
\def\O{{\Omega}}

\def\cK{{\mathcal K}}
\def\cF{{\mathcal F}}
\def\bN{{\mathbf N}}

\def\C{{\mathds C}}
\def\E{{\mathds E}}
\def\N{{\mathds N}}
\def\P{{\mathds P}}
\def\R{{\mathds R}}
\def\V{{\mathds V}{\rm ar\,} }

\def\1{{\mathds 1}}

\def\dint{\textup{d}}

\begin{document}

\title[Identities for Poisson polytopes] {Beyond the Efron-Buchta identities: \\ distributional results for Poisson polytopes}

\author{Mareen Beermann, Matthias Reitzner}
\address{University of Osnabrueck, Department of Mathematics, Albrechtstr. 28a, 49076 Osnabrueck, Germany}
\email{mareen.beermann[at]uni-osnabrueck.de, matthias.reitzner[at]uni-osnabrueck.de}

\begin{abstract}
Let $\Pi$ be a random polytope defined as the convex hull of the points of a Poisson point process. Identities involving the moment generating function of the measure of $\Pi$, the number of vertices of $\Pi$ and the number of non-vertices of $\Pi$ are proven. Equivalently, identities for higher moments of the mentioned random variables are given.

This generalizes analogous identities for functionals of convex hulls of i.i.d points by Efron and Buchta. 
\end{abstract}

\maketitle

\section{Introduction and main results}

Let $\mu$ be some probability measure in $\R^d$ which is absolutely continuous with respect to Lebesgue measure. Choose $m$ random points $X_1, \dots, X_m$ in $\R^d$ independently according to the probability measure $\mu$. We call the convex hull $P_m=[X_1, \dots, X_m] $ of these points a random polytope. Numerous papers have been designated to the study of combinatorial and metric properties of such random polytopes, investigating e.g. the number of facets and the volume.

The problem to determine the expectation $\E N(P_m)$ of the number of vertices of such a random polytope in dimension $d=2$ was first raised by Sylvester precisely 150 years ago in 1864 and so became known as Sylvester's problem. He suggested to choose the points according to Lebesgue measure $\l_2$, naturally restricted to some convex set $K$ of finite area. In the following years a large number of explicit results have been obtained. Most of them concerned the expected area $\E \l_2(P_m)$ of random polygons, where the random points are chosen uniformly in special convex bodies $K$ such as the ellipse  or polygons (see e.g. Buchta \cite{Bu2}, \cite{Bu3}, Buchta and Reitzner \cite{BR2}). Yet, for $d \geq 3$ it appeared to be difficult to evaluate the expected volume for convex bodies different from the unit ball (see Buchta and M\"uller \cite{BM}, Kingman \cite{Ki}, Affentranger \cite{Af1}, Buchta and Reitzner \cite{BR4},  and Zinani \cite{Zi}). Thus, recent developments concentrate on asymptotic results 
as $m \to \infty$.

The question how to link Sylvester's original question asking for the expected number of vertices $\E N(P_m)$ to the expected area, respectively volume $\E \l_d (P_m)$ of the random polytope was answered by Efron \cite{Ef}, who proved for $d=2,3$
$$ \frac{\E \l_d (P_{m})}{\l_d(K)}=1-\frac{\E N(P_{m+1})}{m+1}.
$$
More generally, one can replace Lebesgue measure by some arbitrary probability measure $\mu$ here and obtains
$$\E \mu (P_{m})=1-\frac{\E N(P_{m+1})}{m+1}
$$
for $m$ random points chosen independently according to the probability measure $\mu$. 

For a long time Efron's result -- although frequently used -- stood somehow isolated in the theory of random polytopes. 
Only recently, Buchta \cite{Bu11} was able to complement this equation by identities for higher moments. He proved for $k\in \N$
\begin{equation}\label{eq:Buid}
\E \mu (P_{m})^{k}=\E \prod\limits_{i=1}^{k} \left(1-\frac{N(P_{m+k})}{m+i}\right) . 
\end{equation}
For the first time not only expectations but higher moments of $\l_d (P_m)$, respectively $\mu(P_m)$ were linked to moments of $N(P_m)$. For example, Buchta's identities give rise to an identity for the variances of $\mu(P_m)$ and $N(P_m)$ thus correcting an error in previous results for the variances of these random variables, see \cite{Bu11}.

It is desirable to go a step further by linking the generating functions of $\mu(P_m )$ and $N(P_m)$ and thus the distributions. 
But, to the best of our knowledge, Buchta's identity is still too complicated to lead to a simple identity between the generating functions of $\mu(P_n)$ and $N(P_n)$. 

Yet switching from the binomial model described above to the Poisson model leads to surprisingly simple identities. It is the aim of this paper to state analog's of Buchta's identities in the Poisson model, and then to link the generating functions of $\mu(\cdot)$ and $N(\cdot)$ by an extremely simple identity.

To describe the Poisson model we assume that the number of random points itself is a Poisson distributed random variable $M$ with parameter $t>0$. Then the points $X_1, \dots , X_M$ form a Poisson point process $\eta$ in $\R^d$ of intensity measure $t \mu$. We denote by $\Pi_{t}$ the convex hull of the points of $\eta$. Our main result concerns the number of inner points $I(\Pi_{t}) = M - N(\Pi_t)$ using the (probability-) generating function $g_{I(\Pi_{t})}$ and the moment generating function $h_{\mu(\Pi_{t})}$ of $\mu(\Pi_t)$.
\begin{theorem}\label{th:charfct}
The generating function $g_{I(\Pi_{t})}$ of the number of inner points and the moment generating function $h_{\mu(\Pi_{t})}$ of the $\mu$-measure of $\Pi_t$ are entire functions on $\C$ and satisfy
$$ g_{I(\Pi_{t})}(z+1)=h_{\mu(\Pi_{t})}(tz).
$$
\end{theorem}

This theorem is a consequence of an identity between the moments of $I(\Pi_t)$ and $\mu(\Pi_t)$ and leads to an identity between the cumulants of $I(\Pi_t)$ and $\mu(\Pi_t)$. It is accompanied by a theorem connecting the generating function of the number of vertices to the moment generating function of the $\mu$-measure of $\R^d \backslash \Pi_t$.

The paper is organized in the following way. We provide the  necessary background information and notations in Section \ref{sec:notations}.  Section \ref{sec:resultsI} contains the proof of Theorem~\ref{th:charfct}, the identity concerning the number of inner points of the random polytope and an identity for  the cumulants. The identities relating the generating functions and the moments of the number of vertices of the random polytope are discussed in Section \ref{sec:resultsN}. In Section \ref{sec:appl} some applications to random polytopes in smooth convex bodies are given. 
For the applications we need a lemma, which is also of independent interest. It  is contained in the  Appendix \ref{Appendix}.

For further material on random polytopes we refer to the recent survey articles by Hug \cite{Hugsurv} and Reitzner \cite{Reitsurv}.

\section{Background and Notations}\label{sec:notations}
Let $\mu$ be a probability measure which is absolutely continuous with respect to Lebesgue measure. Assume that $\eta$ is a Poisson point process with intensity measure $t\mu$, $t >0$. The most important examples are given if $\mu$ is either the suitably normalized Lebesgue measure on some convex set $K \subset \R^d$ or the $d$-dimensional Gaussian measure.

More precisely, by $\bN$ we denote the set of all simple and finite counting measures $\nu= \sum \d_{x_i}$ with $x_i \in \R^d$, where simplicity of a counting measure $\nu= \sum \d_{x_i}$ means that $x_i\neq x_j$ for all $i\neq j$. Alternatively, one can think of $\bN$ as the set of all finite point configurations of distinct points in $\R^d$. This can be achieved by identifying the random measure $\nu$ with its support $\{ x_1, x_2 , \dots \}$. Consequently, for $\nu \in \bN$ and a Borel set $A \subset \R^d$, $\nu(A)$ denotes both, the restricted point configuration  $\{ x_1, x_2, \dots \} \cap A$ and the counting measure $\sum \d_{x_i} (A)$. 

Let $(\Omega,\mathcal{F},\P)$ be a probability space. A random measure $\eta: \Omega\to \bN$ is a Poisson point process with intensity measure $t\mu$ if for any Borel set $A$ the random variable $\eta(A)$ is Poisson distributed with parameter $t\mu(A)= \E \eta(A)$, and the random variables $\eta(A_1),\dots,\eta(A_m)$ are independent for pairwise disjoint Borel sets $A_1,\dots,A_m$.

\medskip
By $\Pi_{t}$ we denote the convex hull of the points of $\eta$, which is a random polytope.  $\Pi_{t}^{o}$ will stand for the interior of the random polytope. We will use $N(\Pi_{t})$ for the number of vertices and $I(\Pi_{t})$ for the number of inner points of $\Pi_{t}$, where it holds with probability one that
$$N(\Pi_{t}) = \sum_\eta \1(x \notin \Pi_{t}^o) $$ 
and
$$ I(\Pi_{t})=\eta({\R^d} )-N(\Pi_{t})  = \sum_\eta \1(x \in \Pi_{t}^o).$$
Let us write 
$$\Delta (\Pi_t) = \mu(\R^d \backslash \Pi_t)= 1 -  \mu(\Pi_t) $$
for the $\mu$-content of the complement of $\Pi_{t}$.

\medskip
We will make statements about the (probability-)generating function
$$  g_{X}(z) = \E z^{X} $$
and the moment generating function
$$h_{X}(z) = \E e^{zX} $$
of a random variable $X$ and $z \in \C$.
We set $n_{(k)} = \frac{n!}{(n-k)!}, \; n, k \in \N$.

\medskip
We make use of the Slivnyak-Mecke formula \cite[p.68]{SchnWe3}. In our setting it says that for $m\in \N$ and $f:\bN \times (\R^d)^{m}\rightarrow \R$ a nonnegative measurable function it holds 
\begin{eqnarray}\label{eq:SMF}
\E \lefteqn{\sum_{(x_{1},...,x_{m})\in \eta_{\neq}^{m}}  f(\eta;x_1,...,x_m)}&&
\\ \nonumber &=&
t^m \int\limits_{\R^d}\dots \int\limits_{\R^d} \E f(\eta+\sum_{i=1}^{m}\delta_{x_{i}};x_{1},...,x_{m})
\mu(\dint x_{1})\dots\mu(\dint x_{m}).
\end{eqnarray}
Here $\eta_{\neq}^{m}$ stands for the set of all $m$-tuples of distinct points in $\eta$.

\medskip
Furthermore, we need a relative to the inclusion exclusion principle. Assume $A\subset\R^{d}$, $k\in\N$ and assume $x_{1},...,x_{k}\in\R^{d}$ to be fixed distinct points. Then
\begin{equation}\label{eq:siebf}
 \1(\bigcup_{j=1}^k \{x_j\} \cap A \neq \emptyset) = 
\sum_{r=1}^k (-1)^{r+1} \sum_{I\in \{1, \dots ,k\}^r_{\neq} } 
\1(\bigcup_{j \in I} \{x_j\} \subset A ).
\end{equation}
Here again $\{1, \dots ,k\}^r_{\neq}$ stands for the set of all $r$-tuples of distinct numbers in $\{1, \dots, k\}$.
This formula is just the binomial formula, applied to $(1 - 1)^m$, where $m$ is the cardinality of
$\bigcup_{j=1}^k \{x_j\} \cap A $.

\section{Results for the number of inner points}\label{sec:resultsI}

The aim of this section is to obtain relations between the factorial moments of the number of inner points $I(\Pi_{t})$ and the moments of the $\mu$-content of the random polytope $\Pi_{t}$. From this statement we will deduce Theorem~\ref{th:charfct}.

\begin{theorem}\label{th:ident}
Let $I(\Pi_{t})$ be the number of inner points and $\mu(\Pi_{t})$ the $\mu$-content of the random polytope $\Pi_{t}$. Then for $k \in \N$
$$ \E I(\Pi_{t})_{(k)} = t^k\E \mu(\Pi_{t})^k . 
$$
\end{theorem}

We make this explicit in the particular cases $k=1,2$.
For $k=1$, Theorem~\ref{th:ident} yields for the expectations of these random variables
\begin{equation}\label{eq:idE}
 \E I(\Pi_{t}) = t\E \mu(\Pi_{t}) .
\end{equation}
For $k=2$ we obtain an identity for the variances,
\begin{equation}\label{eq:idVar}
\V I(\Pi_{t}) = t^2\V \mu(\Pi_{t}) + t \E \mu(\Pi_{t}).
\end{equation}

\begin{proof}[of Theorem~\ref{th:ident}]
Consider the number of inner points $I(\Pi_{t})$,
$$
I(\Pi_{t})=\sum_{x \in \eta}\1(x\in \Pi_{t}^o).
$$
The number of (ordered) $k$-tuples of pairwise distinct inner points of $\Pi_{t}$ is given by $I(\Pi_{t})_{(k)}$.
To calculate the expected value of this, we use for a point set $\xi$ the notation $[\xi]$ for the convex hull of the points in $\xi$
and apply the Slivnyak-Mecke formula (\ref{eq:SMF}).
\begin{eqnarray*}
\E I(\Pi_{t})_{(k)}
&=&
\E\sum_{(x_{1}, \dots ,x_{k})\in \eta^k_{ \neq}} \prod\limits_{j=1}^{k} \1(x_{j}\in [\eta]^o) 
\\ &=&
t^k \E \int\limits_{{\R^d} } \dots \int\limits_{{\R^d} } \prod\limits_{j=1}^{k} 
\1(x_{j}\in [\eta, x_1, \dots x_k]^{o})\; \dint \m(x_{1})...\dint \m(x_{k})
\\ &=&
t^k \E  \int\limits_{{\R^d} } \dots \int\limits_{{\R^d} }  \prod\limits_{j=1}^{k} \1(x_{j}\in [\eta]^o) \; \dint \m(x_{1}) \dots \dint \m(x_{k})
\\ &=&
t^k\E\m(\Pi_{t}^o)^k
\end{eqnarray*} 
Since $\mu$ is absolutely continuous,  Theorem~\ref{th:ident}  follows.
\qed
\end{proof}

This identity leads to the relation between the generating function of the number of inner points and the moment generating function of the $\mu$-content of the random polytope $\Pi_{t}$, as already stated in \\[1ex]
{\bf Theorem~\ref{th:charfct}.}
{\it The generating function $g_{I(\Pi_{t})}$ of the number of inner points and the moment generating function $h_{\mu(\Pi_{t})}$ of the $\mu$-measure of $\Pi_t$ are entire functions on $\C$ and satisfy
$$
g_{I(\Pi_{t})}(z+1)=h_{\mu(\Pi_{t})}(tz).
$$
}

\begin{proof}
Recall that the generating function of the inner points is given by
$$g_{I(\Pi_{t})}(z) = \E z^{I(\Pi_{t})} = \sum_{k=0}^{\infty} z^k \P( I(\Pi_{t})=k)  .$$ 
For $|z|<1$ the generating function is always absolutely convergent. Since $I(\Pi_{t}) \leq \eta (\R^d)$,  we also have for $|z| \geq 1$
$$
| z^{I(\Pi_{t})} | \leq |z^{\eta(\R^d)}|.
$$
This implies 
$$
| \E z^ {I(\Pi_{t})} | \leq \E |z^ {\eta(\R^d)}| = \sum_{k=1}^{\infty} |z|^k e^ {-t} \frac{t^k}{k!} = e^ {t(|z|-1) } < \infty
$$ 
because $\eta(\R^d)$ is Poisson distributed with parameter $t$.
Hence, $g_{I(\Pi_{t})}$ is an entire function on $\C$.

It is well known that if $g_{I(\Pi_{t})}$ is an entire function, the $k$-th derivatives of $g_{I(\Pi_{t})}$ at the point $z=1$ are the $k$-th factorial moments of $I(\Pi_{t})$.
$$
g_{I(\Pi_{t})}^{(k)} (1)=\E I(\Pi_{t})(I(\Pi_{t})-1)\cdot \dots \cdot(I(\Pi_{t})-k+1)z^{I(\Pi_{t})-k}|_{z=1}=\E I(\Pi_{t})_{(k)}
$$
We evaluate the analytic function $g_{I(\Pi_{t})} (z+1)$ at $z=0$ and deduce
\begin{equation}\label{eq:gIreihe}
g_{I(\Pi_{t})}(z+1)
=\sum_{k=0}^{\infty}g_{I(\Pi_{t})}^{(k)} (1)\,  \frac{z^k}{k!}
=\sum_{k=0}^{\infty}\E I(\Pi_{t})_{(k)}\ \frac{z^k}{k!}.
\end{equation}

Since the random variable $\mu(\Pi_{t}) $ is bounded by $ \mu(\R^d)=1 $, the moment generating function of $\mu(\Pi_{t})$ is also an entire function. Its derivatives at $z=0$ are given by the moments of $\mu(\Pi_{t})$.
$$
h_{\mu(\Pi_{t})}^{(k)} (0)=\E \mu(\Pi_{t})^{k} e^{z \mu(\Pi_{t})}|_{z=0}=\E \mu(\Pi_{t})^k
$$
Because $h_{\mu(\Pi_{t})} (z)$ is an entire function and thus analytic, we can write 
\begin{equation}\label{eq:hmreihe}
h_{\mu(\Pi_{t})}(z) =
\sum_{k=0}^{\infty} h_{\mu(\Pi_{t})}^{(k)} (0) \frac{z^k}{k!} 
=
\sum_{k=0}^{\infty}  \E {\mu(\Pi_{t})}^{k}\  \frac{z^k}{k!} .
\end{equation}
Combining (\ref{eq:gIreihe}) and (\ref{eq:hmreihe}) with Theorem~\ref{th:ident} proves Theorem~\ref{th:charfct}. \qed
\end{proof}

In the next step we use this relation between the moment generating function of $\mu(\Pi_t)$ and the generating function of $I(\Pi_t)$ to prove a relation between their cumulants. First recall that the cumulant generating function of a random variable $X$ is given by
$$\ln h_{X}(z) = \ln  \E e^{zX} =  \sum_{k=1}^\infty \kappa_k \frac{t^k}{k!} , $$ 
where $\kappa_k$ is the cumulant of order $k$.
Due to Theorem~\ref{th:charfct} we have 
\begin{equation}\label{eq:lnhI} 
\ln h_{\mu(\Pi_{t})}(tz) = \ln g_{I(\Pi_{t})}(z+1)  =  \ln h_{I(\Pi_{t})}(\ln(z+1))  . 
\end{equation}
Essential for the relation between the cumulants of the moment generating function of $\mu(\Pi_t)$ and the generating function of $I(\Pi_t)$ are the Stirling numbers of the first kind defined by the expansion of the function \linebreak
$z_{(n)} = z (z-1) \dots (z-n+1)$ for $n\in \N$ into a power series in $z$,
$$z _{(n)} =\sum\limits_{k=1}^{n}\begin{bmatrix} n \\ k \end{bmatrix} z^k . $$
The Stirling numbers of the first kind satisfy (or can equivalently be defined by)
\begin{equation}\label{eq:Stirlingln}
 \frac{\ln^j (z+1)}{j!} =\sum\limits_{k=j}^{\infty} \begin{bmatrix} k \\ j \end{bmatrix} \frac {z^k} {k!} .
\end{equation}

\begin{theorem}\label{th:cumulants}
Let $\kappa_k(\mu(\Pi_t))$, resp. $\kappa_k(I(\Pi_t))$ be the cumulants of the $\mu$-measure $\mu(\Pi_t)$, resp. of the number of inner points $I(\Pi_t)$. Then
$$
t^k \kappa_k (\mu(\Pi_t)) = 
\sum\limits_{j=1}^{k}\begin{bmatrix} k \\ j \end{bmatrix} \kappa_j(I(\Pi_t)) .
$$
\end{theorem}

\begin{proof}
By definition of the cumulants and because of (\ref{eq:lnhI}) we have
\begin{eqnarray} \label{eq:cum1}
\sum_{k=1}^\infty t^k  \kappa_k (\mu(\Pi_t)) \frac{ z^k}{k!}  
&=&  
\ln h_{\mu(\Pi_{t})}(tz)  
\\ &=& \nonumber
\ln h_{I(\Pi_{t})}(\ln(z+1)). 
\end{eqnarray}
We expand the last expression in a series in $\ln(z+1)$ with coefficients given by the cumulants of $I(\Pi_t)$.
\begin{eqnarray*}
\ln \E e^{\ln(z+1) I(\Pi_t)}
 &=&
\sum_{j=1}^\infty \kappa_j(I(\Pi_t)) \frac{\ln^j (z+1)}{j!}   
\end{eqnarray*}
Using property (\ref{eq:Stirlingln}) for the logarithmic term gives
\begin{eqnarray} \label{eq:cum2} \nonumber
\sum_{j=1}^\infty \kappa_j(I(\Pi_t)) \frac{\ln^j (z+1)}{j!}   
&=&
\sum_{j=1}^\infty \kappa_j(I(\Pi_t))  \sum\limits_{k= j}^{\infty}  \begin{bmatrix} k \\ j \end{bmatrix} \frac{z^k}{k!} 
\\&=& 
\sum_{k=1}^\infty 
\Big( \sum\limits_{j=1}^{k}\begin{bmatrix} k \\ j \end{bmatrix} \kappa_j(I(\Pi_t)) \Big)\ 
\frac{z^k}{k!}. 
\end{eqnarray}
Comparing coefficients of $\frac{z^k}{k!}$ in (\ref{eq:cum1}) and (\ref{eq:cum2}) proves our theorem. \qed
\end{proof}

\section{Results for the number of vertices}\label{sec:resultsN}

Analogously to Theorem~\ref{th:charfct} we want to state a theorem connecting the measure of the missed set 
$\Delta(\Pi_{t})=\mu(\R^d \setminus \Pi_{t})$ 
and the number of vertices $N(\Pi_{t})$. Moreover, we find a relation between higher moments of these two variables. However, the relation in this case is not that immediate as the identity in the case of the inner points of $\Pi_{t}$.

\begin{theorem}\label{th:mgfctN2}
The generating function $g_{N(\Pi_{t})}$ of the number of vertices and the moment generating function $h_{\Delta(\Pi_{t})}$ of the $\mu$-measure of $\R^d \backslash \Pi_t$ satisfy for $x \in [0,1]$
$$
g_{N(\Pi_t)} (x) =
h_{\Delta(\Pi_{xt})} (t(x-1)).
$$
\end{theorem}

Before giving the proof of this theorem, we compare Theorem~\ref{th:charfct} to Theorem~\ref{th:mgfctN2}. Substituting $z$ by $z-1$ in the first mentioned theorem, the statements of these Theorems read as
\begin{eqnarray} \label{eq:gIhmu}
g_{I(\Pi_{t})}(z) &=& h_{\,\mu(\Pi_{t})\,}(t(z-1)),
\\ \nonumber
g_{N(\Pi_t)} (x) &=&  h_{\Delta(\Pi_{xt})} (t(x-1)).
\end{eqnarray}
 The main difference is the occurrence  of $x$ in the random variable $\Delta(\Pi_{xt})$ in the second line, which makes sense only if $x$ is in $\R_+$ and makes it impossible to extend the right hand side to an holomorphic function. It would be of interest to deduce one of these identities from the other, but we have been unable to find a connection.

It should be remarked that it is possible to prove the identity (\ref{eq:gIhmu}) for $z \in [0,1]$ using the method applied in the proof of Theorem~\ref{th:mgfctN2}. By the identity theorem for holomorphic functions we could deduce that equality holds for all $z \in \C$ because $g_{I(\Pi_{t})}$ and $h_{\,\mu(\Pi_{t})\,}$ are entire functions.
It is straightforward to prove that also $g_{N(\Pi_t)}(z)$ and $h_{\Delta(\Pi_{t})}(z)$ are both entire functions, but we make no use of this fact in our investigations.

\begin{proof}[of Theorem~\ref{th:mgfctN2}]
Suppose $\eta_{xt}$ and $\bar{\eta}_{yt}$ are two independent Poisson point processes on ${\R^d}  $ with intensity measure $xt \mu$, resp. $yt \mu$ with $x,y\geq 0,\ x+y=1$. It is well known that 
$$\eta \stackrel{d}{=} \eta_{xt} + \bar{\eta}_{yt} .$$
Conversely, if we split $\eta$ into two point sets by deciding for each point of $\eta$ independently if it belongs to $\eta_1$ with probability $x$ or to $\eta_2$ with probability $y=1-x$, then $\eta_1$, resp. $ \eta_2$ equals $\eta_{xt}$, resp. $ \bar{\eta}_{yt}$ in distribution.

Denote by $\cF_N(\Pi_t)$ the set of vertices of $\Pi_t$. As described above we split $\eta$ into $\eta_{xt}$ and $\bar \eta_{yt}$ and consider the event that all vertices of $\Pi_{t}$ emerge from $\eta_{xt}$. This event occurs if no point of $\bar{\eta}_{yt}$ is contained in ${\R^d}  \setminus \Pi_{xt}$, where $\Pi_{xt}$ is the convex hull of the points of $\eta_{xt}$. Because  these point processes are independent, we have
\begin{equation}\label{eq:hD}
\P(\cF_N(\Pi_{t}) \subset \eta_{xt} )
=
\P(\bar{\eta}_{yt} ({\R^d}  \setminus \Pi_{xt} )=0) 
=
\E (e^{-yt \Delta(\Pi_{xt}) }).
\end{equation}
Moreover, to compute $\P(\cF_N(\Pi_{t}) \subset \eta_{xt} )$ we first condition on the number of vertices
$$
\P(\cF_N(\Pi_{t}) \subset \eta_{xt} | N (\Pi_{t}) =k) =  x^{k},
$$
which follows from the splitting argument stated above. Taking expectation and thus removing the condition, we get
\begin{equation} \label{eq:gN}
\sum\limits_{k=0}^{\infty}  x^{k} \P(N(\Pi_{t}) =k)= \E x^{N(\Pi_{t})}.
\end{equation}

Combining (\ref{eq:hD}) and (\ref{eq:gN}) yields our Theorem.  \qed
\end{proof}

Theorem~\ref{th:mgfctN2} states the relation between the factorial moment generating function of the number of vertices and the moment generating function of the $\mu$-content of the missed set $\R^d \setminus \Pi_{xt}$. 
Due to the occurrence of $x$ in the random variable $\Delta(\Pi_{xt})$ it seems impossible to state a simple identity between factorial moments of $N(\Pi_t)$ and $\Delta(\Pi_t)$.
As can be seen in the next theorem, there is a much more complicated relation for the moments of these two random variables.
Again we use the notation $[\xi]$ for the convex hull of points of a point set $\xi$.

\begin{theorem}\label{th:N}
Let $N(\Pi_{t})$ be the number of vertices and $\Delta(\Pi_{t})$ the $\mu$-content of the complement of $\Pi_{t}$. Then for $k \in \N$
\begin{eqnarray*}
\E N(\Pi_t)_{(k)}
&=&
t^k \, \E \Delta(\Pi_t)^{k}
-  t^k\sum_{r=1}^{k-1} (-1)^{r+1} {\binom k r} 
\\ \nonumber && \times
\E \int\limits_{{\R^d}  \setminus \Pi_t} \! \dots \! \int\limits_{{\R^d}  \setminus \Pi_t} \mu([\eta, x_1, \dots, x_{k-r}] \setminus [\eta])^r \, 
\dint\mu(x_1) \dots \dint\mu(x_{k-r}) .
\end{eqnarray*}
\end{theorem}

The particular case $k=1$ gives a simple identity for the expected values
\begin{equation}\label{eq:IdNk1}
\E N(\Pi_t) = t \E \Delta (\Pi_t) .
\end{equation}
And for $k=2$ we obtain the more complicated expression
\begin{equation}\label{eq:IdNk2}
\E N(\Pi_t)_{(2)}=
t^2 \E \Delta(\Pi_t)^{2}
- 2 t^2\E \int\limits_{{\R^d}  \setminus \Pi_t} \mu( [\eta,x] \setminus [\eta])  
\, \dint \mu(x). 
\end{equation}
Formulas (\ref{eq:IdNk1}) and (\ref{eq:IdNk2}) can be used to deduce for the variances the relation 
\begin{equation}\label{VarN}
\V N(\Pi_t)  =  t^2 \V \Delta(\Pi_t)  + t \E \Delta(\Pi_{t}) 
- 2 t^2\E \int\limits_{{\R^d}  \setminus \Pi_t} \mu( [\eta,x] \setminus [\eta])  
\, \dint \mu(x). 
\end{equation}

\begin{proof}[of Theorem~\ref{th:N}]
We are interested in the factorial moments of the number of vertices $N(\Pi_t)=\sum\1(x\notin \Pi_t^{o}) $ of the random polytope $\Pi_{t}$.
We apply the Slivnyak-Mecke formula (\ref{eq:SMF}) to obtain
\begin{align*}
\E N(\Pi_t)_{(k)}
& = \E \sum_{(x_{1}, \dots ,x_{k}) \in \eta^{k}_{\neq}}\1(x_{1}\notin \Pi_{t}^{o}) \dots \1(x_{k}\notin \Pi_{t}^{o})\\ 
& = 
t^k \ \E \int\limits_{\R^d}  \dots \int\limits_{\R^d} 
\prod_{j=1}^k \1(x_{j}\notin [\eta,x_1, \dots, x_k]^{o})  \, \dint \mu(x_1) \dots \dint \mu(x_{k}).
\end{align*}
To go further we have to evaluate the occurring product. For this we make use of formula (\ref{eq:siebf}) with 
$$
A  = [\eta, x_1, \dots , x_k]^{o} \setminus [\eta]^o,
$$
that is
\begin{align*}
&
\1(\bigcup_{j=1}^k \{x_j\} \cap [\eta, x_1, \dots , x_k]^{o} \setminus [\eta]^o \neq \emptyset)
\\ &= 
\sum_{r=1}^{k-1} (-1)^{r+1} \sum_{I\in \{1, \dots ,k\}^r_{\neq} } 
\1(\bigcup_{j \in I} \{x_j\} \subset [\eta, x_1, \dots , x_k]^{o} \setminus [\eta]^o ).
\end{align*}
Because it is impossible that all points $\{ x_1, \dots, x_k\}$ are in  $[\eta, x_1, \dots , x_k]^{o} \setminus [\eta]^o$, the term for $r=k$ is missing.
If we multiply both sides by $ \prod_{j=1}^k \1(x_j \notin [\eta]^o)$ we obtain 
\begin{eqnarray*}
\prod_{j=1}^k \1(x_{j} \lefteqn{\notin [\eta, x_1, \dots, x_k]^{o})} &&
\\  &=& 
\prod_{j=1}^k \1(x_j \notin [\eta]^o)  
\1(x_j \notin [\eta, x_1, \dots , x_k]^{o} \setminus [\eta]^o )
\\  &=& 
\prod_{j=1}^k \1(x_j \notin [\eta]^o) 
\Big(1- \1(\bigcup_{j=1}^k \{x_j\} \cap [\eta, x_1, \dots , x_k]^{o} \setminus [\eta]^o \neq \emptyset) \Big) 
\\ &=& 
\prod_{j=1}^k \1(x_j \notin [\eta]^o) -
\sum_{r=1}^{k-1} (-1)^{r+1} 
\\ && \hskip1cm 
\sum_{I\in \{1, \dots ,k\}^r_{\neq} } \ 
\prod_{j \in I} \1(x_j\in [\eta, x_1, \dots , x_k]^{o} \setminus [\eta]^o) 
\prod_{j=1}^k \1(x_j \notin [\eta]^o) .
\end{eqnarray*}
In the next step we have to integrate over all $x_1, \dots , x_k $ in $ \R^d $ or, more precisely, over $ \R^d \setminus \Pi_t$ because of the indicator functions $\1(x_j \notin [\eta]^o)$. The integral of the first term on the right side thus equals $ \Delta(\Pi_{t})^k$. We obtain
\begin{eqnarray*}
\E N(\Pi_t)_{(k)}
&=&
t^k \ \E \Delta(\Pi_t)^k 
- t^k\ \sum\limits_{r=1}^{k-1} (-1)^{r+1}  \sum_{I\in \{1, \dots ,k\}^r_{\neq} }
\\ && 
 \E 
\int\limits_{{\R^d}  \setminus \Pi_t} \!\!\! \dots \!\!\! \int\limits_{{\R^d}  \setminus \Pi_t} 
\prod_{j \in I} \1(x_j\in [\eta, x_1, \dots , x_k]^{o} \setminus [\eta]^o) \dint \mu(x_1) \dots \dint \mu(x_k)
\\ &=&
t^k \ \E \Delta(\Pi_t)^k 
- t^k\ \sum\limits_{r=1}^{k-1} (-1)^{r+1}  \binom k r
\\ && 
 \E 
\int\limits_{{\R^d}  \setminus \Pi_t} \dots \int\limits_{{\R^d}  \setminus \Pi_t} 
\mu([\eta, x_1, \dots , x_{k-r}] \setminus [\eta])^r \dint \mu(x_1) \dots \dint \mu(x_{k-r}).
\end{eqnarray*}
\qed \end{proof}

\section{Applications}\label{sec:appl}

In the last thirty years many papers have been devoted to compute the asymptotic distribution of the quantities mentioned above, in many cases for the Poisson model and under the assumption that $\mu$ is the uniform distribution on a smooth convex set or a polytope, or for the $d$-dimensional Gaussian measure. Most of this results carry over to the binomial model by some de-Poissonization arguments, see e.g. the papers by Calka and Yukich \cite{CaYu} and B{\'a}r{\'a}ny and \linebreak Reitzner \cite{BRe1,BRe2}.

In this last chapter we want to contribute to these results giving an example of how our results can be applied. 
Assume that $K \in \cK^k_+$, i.e. it has $k$-times continuously differentiable boundary of positive Gaussian curvature and volume one. 
Let $\mu(\cdot)=\l_d(K \cap \cdot)$ be Lebesgue measure restricted to the convex body $K$, and hence $\Pi_t$ is a Poisson polytope inscribed in $K$. 
After planar results going back to Renyi and Sulanke (\cite{RS1} and \cite{RS2}) it was shown by B{\'a}r{\'a}ny \cite{Bar2} that for any $d$-dimensional smooth convex body $K \in \cK^3_+$
\begin{eqnarray}
\label{eq:EN}
\E N(\Pi_{t}) &=& c_1 \Omega(K)t^{\frac{d-1}{d+1}}+o(t^{\frac{d-1}{d+1}}),
\\ \nonumber
\l_d(K) - \E \l_d (\Pi_t) &=&  c_1\Omega(K)t^{-\frac{2}{d+1}}+o(t^{\frac{d-1}{d+1}}) 
\end{eqnarray}
as $t\to \infty$, and where $\Omega(K)$ denotes the affine surface area of the boundary of $K$.
In fact, these results have been obtained for the binomial model, but it is easy to see that results for the binomial model immediately carry over to the Poisson model.
For a long time it was out of reach to compute the asymptotic behavior of the variance or even precise estimates. Only recently it was proved by Reitzner \cite{Re6}, using the Efron-Stein jackknife inequality, that for $K \in \cK^2_+$ there are constants $c_2(K),c_3(K)>0$ such that 
\begin{eqnarray*}
c_2(K) t^{- \frac {d+3}{d+1}}& \leq \V \l_d(\Pi_t) \leq & c_3(K) t^{- \frac {d+3}{d+1}}  ,\
 \\
c_2(K) t^{ \frac {d-1}{d+1}}  &\leq \V N(\Pi_{t}) \leq& c_3(K) t^{ \frac {d-1}{d+1}} .
\end{eqnarray*}
A very recent breakthrough was achieved by Calka and Yukich, who calculated in \cite{CaYu} the precise asymptotics for the variances of the number of vertices and the volume of the random polytope $\Pi_{t}$. We have for $K\in \mathcal{K}_{+}^{3}$
\begin{align}\label{eq:VarNYukich,Calka}
\V N(\Pi_{t})&=c_{4}\Omega(K)t^{\frac{d-1}{d+1}}+o(t^{\frac{d-1}{d+1}}) 
\end{align}
and for $K\in \mathcal{K}_{+}^{6}$
\begin{align}\label{eq:VarVYukich,Calka}
\V \l_d (\Pi_{t})&=c_{5}\Omega(K)t^{-\frac{d+3}{d+1}}+o(t^{-\frac{d+3}{d+1}}) 
\end{align}
as $t\to\infty$.
We can apply our identities to deduce one from the other. 
Because $\V \Delta(\Pi_{t})=\V \l_d(\Pi_{t})$, equation (\ref{VarN}) implies
\begin{align*}
t^2 \V \l_d(\Pi_t) 
&=  
\V N(\Pi_{t}) -  \E N(\Pi_{t}) 
+ 2 t^2\E \int\limits_{K \setminus \Pi_t} \l_d( [\eta,x] \setminus [\eta])  
\, \dint x \, .
\end{align*}
By (\ref{eq:VarNYukich,Calka}), and by (\ref{eq:EN}) it follows for $K\in \mathcal{K}_+^{3}$
$$
\V \l_d(\Pi_{t})=c_{6}\Omega(K)t^{-\frac{d+3}{d-1}}+o(t^{-\frac{d+3}{d-1}})
-2\E \int\limits_{K \setminus \Pi_t} \l_d([\eta,x] \setminus [\eta])  
\, \dint x
$$
as $t\to\infty$.
We will prove in the appendix that 
$D_{t}=\int\limits_{K \setminus \Pi_t} \l_d([\eta,x] \setminus [\eta]) \, dx$
satisfies
$$ \E D_t =  c_{7}   \Omega(K)  t^{-\frac {d+3} {d+1} } + o(t^{  -\frac {d+3} {d+1} }) $$
for $K \in \cK^2_+$ as $t\to\infty$. Combining these estimates proves the following corollary.
\begin{corollary}
For $K \in \cK^3_+$ we have
$$
\V \l_d(\Pi_{t})=c_{8}\Omega(K)t^{-\frac{d+3}{d-1}}+o(t^{-\frac{d+3}{d-1}})
$$
as $t\to\infty$.
\end{corollary}
This is the result of Calka and Yukich \cite{CaYu} for a slightly bigger class of convex bodies. As in their paper this could be transferred to a formula giving the asymptotic variance for the binomial model.

Furthermore, we can use this corollary, (\ref{eq:EN}) and (\ref{eq:idVar}) to obtain asymptotically the variance of the number of inner points $I(\Pi_{t})$ for $K\in \mathcal{K}^{3}_{+}$. 
\begin{align*}
\V I(\Pi_{t})
&=t^{2}\V \l_d(\Pi_{t})+t\E \l_d(\Pi_{t})\\
&=t+c_{9}\Omega(K)t^{\frac{d-1}{d+1}}+o(t^{\frac{d-1}{d+1}}) 
\end{align*}
as $t\to\infty$.
Observe that it follows immediately from $N(\Pi_t)+ I(\Pi_t)= \eta(K)$ that 
$$
\E I(\Pi_{t}) = 
t - c_{10}\Omega(K)t^{\frac{d-1}{d+1}}+o(t^{\frac{d-1}{d+1}})
$$
as $t\to\infty$. 

Similarly one could apply our identities in the case when the intensity measure of the Poisson point process is a multiple of the uniform measure on a polytope $K$, or a multiple of the Gaussian distribution. We refer to \cite{BRe2} and \cite{HMR}, and leave the details to the interested reader.

\section{Appendix}\label{Appendix}

\begin{theorem}\label{th:appendix}
Assume that $K \in {\mathcal K}^{2}_+$ with $\l_d (K)=1$, and let $\Pi_t = [\eta]$ be the Poisson polytope chosen according to the intensity measure $t \l_d (K \cap \cdot) $. 
Define 
$$D_{t}=\int\limits_{K \setminus \Pi_t} \l_d([\eta,x] \setminus [\eta]) \, dx . $$ 
Then there is a positive constant ${\rm C}_d$ depending on the dimension such that
$$ 
\lim\limits_{t \to \infty} \E D_t t^{ 1 +\frac {2} {d+1} }=  {\rm C}_d   \Omega(K) .
$$
\end{theorem}

For $x \in K$ denote by $\cF (\eta, x)$ the set of facets of $\Pi_t$ which can be seen from $x$, 
i.e. which are facets of $\Pi_t$ but not of $[\Pi_t, x]= [\eta, x]$.
Note that this set is empty if $x \in \Pi_t$. 
Using this notation we have
$$
\l_d([\eta,x] \setminus [\eta])
=
\frac 1{d!}\sum_{(x_1, \dots, x_d) \in \eta_{\neq}^{d} } \1([x_1, \dots, x_d] \in \cF(\eta, x) ) \l_{d}[x_1, \dots, x_d,x] .
$$
The Slivnyak-Mecke formula (\ref{eq:SMF}) yields
\begin{eqnarray*}
\E D_t
&=&
\frac1{d!} \int \limits_K \E \sum_{(x_1, \dots, x_d) \in \eta_{\neq}^{d} } \1([x_1, \dots, x_d] \in \cF(\eta, x) ) \l_{d}[x_1, \dots, x_d,x] \,  \dint x 
\\&=&
\frac1{d!}  t^d \int \limits_K \dots \int \limits_K  \E \1(F \in \cF(\eta + \sum \d_{x_i}, x) ) \l_{d}[F,x] \, \dint x_1 \dots \dint x_d  \dint x,
\end{eqnarray*}
where $F= [x_1, \dots, x_d]$.
The affine hull of $F$ is a hyperplane which cuts $K$ into two parts. Denote by $K_+(F)$ that part of $K$ which contains $x$.
The indicator function equals one if the affine hull of $F$ separates $x$ from $\eta$, i.e. if $\eta (K_+)=0$. This happens with probability $e^{- t \l_{d}(K_+(F))}$.
The volume of the simplex $[F,x]$ equals $1/d$ times the base $\l_{d-1}(F)$ times the height, which is the distance $d_{{\rm aff} F}(x)$ of $x$ to the affine hull of $F$.
$$  \E D_t = \frac 1{d\, d!}  t^d  \int \limits_K \cdots \int \limits_K e^{-t\l_{d}(K_+(F))}   \l_{d-1}(F) d_{{\rm aff} F}(x)\,   dx_1 \cdots dx_d dx $$
The next lemma gives the asymptotic behavior of this integral and thus proves our theorem.

\begin{lemma}\label{LemmaAppendix}
Assume that $K \in {\mathcal K}^{2}_+$ with $ \l_{d}(K)=1$. Then 
\begin{eqnarray} \label{asymp}
\int \limits_K \cdots \int \limits_K e^{-t\l_{d}(K_+(F))}  \lefteqn{ \l_{d-1}(F) d_{{\rm aff} F}(x)\,   dx_1 \cdots dx_d dx }
&&
\\ & = & 
c_d   \Omega(K)   t^{-(d+1)-\frac 2 {d+1}} + o \left( t^{-d-1-\frac 2 {d+1}} \right)
\end{eqnarray}
as $t \to \infty$.
\end{lemma}

Principal ideas for the proof of this lemma are taken from \cite{Re6}, where the asymptotics of a similar integral was computed.

\begin{proof}
In a first step we transform the integral using the Blaschke--Petkantschin formula (cf., e.g., \cite[p.278]{SchnWe3}),
\begin{eqnarray*}\label{blaschke}
\int \limits_K   \lefteqn{ \cdots  \int \limits_K  f(x_1, \dots, x_d)  \dint x_1 \cdots \dint  x_d }&&
\\ &=&  
(d-1)!  \int \limits_{H \in {\mathcal H} (d,d-1)}  \int \limits_{K\cap H}  \cdots  \int \limits_{K\cap H}
f(x_1, \dots, x_d)  \l_{d-1} ( F)   \dint x_1  \cdots \dint x_d   \dint H .
\end{eqnarray*}
The differential $dH$ corresponds to the suitably normalized rigid motion invariant Haar measure on the
Grassmannian ${\mathcal H} (d,d-1)$ of hyperplanes in $\R^d$. A hyperplane is given by its unit
normal vector $u \in S^{d-1}$ and its signed distance $h$ to the origin, $H= \{ y:\ \langle y,u\rangle =h\}$. Let $H_+ =\{ y:\  \langle y,u \rangle \geq h\}$ be the corresponding halfspace. Denoting by $du$ the element of surface area on $S^{d-1}$, we have $dH = \frac 12 dh du$, $u \in S^{d-1}, h \in \R$. (Observe that $H(h,u)=H(-h,-u)$, which explains the factor $\frac 12$.)

Because of $\l_{d-1} (F)$ the integrand vanishes outside the interval 
$$h \in [- h_K(-u),h_K (u) ], $$
where $h_K (u)$ is the support function of $K$ in direction $u$. Given $H=H(h,u)$, we assume that the additional point $x \in H_+$. Then $K_+(F)= K \cap H_+$ and $\l_+= \l_{d}(K_+)$ only depends on $H_+$ but not on the relative position of the points $x_j \in H$. 
This yields
\begin{eqnarray}\nonumber
 \int \limits_K \cdots  \lefteqn{\int \limits_K e^{-t\l_{d}(K_+(F))}  \l_{d-1}(F) \dint_{{\rm aff} F}(x)\,   \dint x_1 \cdots \dint x_d \dint x }&&
\\ & =& \label{eq:inth}
\frac {(d-1)!}2
\int \limits_{S^{d-1} } \int \limits_{- h_K(-u)}^{h_K (u)}  e^{-t \l_+}  {\mathcal I}_{K \cap H} 
{\mathcal J}_{K \cap H_+} \, \dint h \dint u
\end{eqnarray}
with
$$ {\mathcal I}_{K \cap H}= \int \limits_{K\cap H} \cdots \int \limits_{K\cap H}  \l_{d-1}( F)^2 \
 \dint  x_1  \cdots \dint x_d , \ \ 
 {\mathcal J}_{K \cap H_+} = 
\int \limits_{K_+} \dint_{H}(x) \, \dint x  .
 $$

Given some $\e>0$, we split the integral in (\ref{eq:inth}) with respect to $h$ into two parts: $h \in [-h_K(-u), h_K(u)- \e]$ and $h \in [h_K(u)-\e, h_K(u)]$.
Estimating the integral
$$ \int \limits_{-h_K(-u)}^{h_K (u)-\e}   e^{-t\l_+}   {\mathcal I}_{K \cap H} {\mathcal J}_{K \cap H_+} \ \dint h $$
is easy. The integrals ${\mathcal I}_{K \cap H}$ and ${\mathcal J}_{K \cap H_+} $ are always bounded by a constant $\gamma_1$ independent of $h$ and $u$. There exists a constant $\gamma_2 = \gamma_2 (\d) >0$ independent of $u$ with $\l_+ =\l_+(h,u) \geq \gamma_2$. And $h_K(u)+h_K(-u) $ is bounded by some constant $\gamma_3$ independent of $u$. Thus for $h \leq h_K(u)-\e$ we have
\begin{equation} \label{gamma}
0 \leq \int \limits_{- h_K(-u)}^{h_K (u)-\e}   e^{-t\l_+}  {\mathcal I}_{K \cap H} {\mathcal J}_{K \cap H_+} \ \dint h \leq \gamma_1^2 \gamma_3 e^{-t\gamma_2} .
\end{equation}

We estimate the second part of the integral. Let $u \in S^{d-1} $ be fixed. As $K$ is of class ${\mathcal K}^2_+$, there is an unique point $p \in \partial K$ with outer normal vector $u$. Choose $\d >0$ sufficiently small. There exists a paraboloid $q^{(p)}(y)$ and a $\l =\l(\d) >0$ such that the $\l$-neighborhood of $p$ in $\partial K$ can be represented by a convex function $f^{(p)}(y)$ fulfilling 
\begin{equation} \label{eq:q}
((1+\d)^{-1}  q^{(p)}(y) +p)  \leq f^{(p)} (y) \leq ((1+\d)  q^{(p)}(y) +p) . \end{equation}
Now we fix $\e>0$ such that for each $u$ the intersection $H(h_K(u)- \e, u) \cap \partial K$ is contained in this $\l$-neighborhood of the boundary point $p$.

Let $\R^d = \{(y,z) \vert y \in \R^{d-1},   z \in \R \} $.
For the moment identify the tangent hyperplane to $\partial K$ at $p$ with the plane $z =0$ and $p$ with the origin such that
$K$ is contained in the halfspace $z \geq 0$ and $u$ coincides with $(0,-1)$. Hence, in this situation $h_K(u)=0$. Define $H(z)=H(-h, u)$ to be the hyperplane parallel to $z=0$ with distance $z$ to the origin, and in accordance with the definition above, $H_+(z)$ to be the corresponding halfspace containing the new origin.

We introduce polar coordinates: let $\R^d = (\R^+ \times S^{d-2}) \times \R$ and denote by $(r v,z)$ a point in $\R^d$,
$r \in \R^+$, $v \in S^{d-2}$, $z \in \R$.
Since $K \in {\mathcal K}^2_+$, by choosing a suitable Cartesian coordinate system in $\R^{d-1}$, the paraboloid can be  parametrized by
$$ b_2 (rv)=\tfrac 12 ( k_1 \langle rv,e_1 \rangle ^2+ \cdots +k_{d-1} \langle rv,e_{d-1} \rangle ^2), $$
where $k_1, \dots, k_{d-1}$ are the principal curvatures of $K$ at $p$.
The estimate (\ref{eq:q}) reads as 
$$ (1+\d)^{-1}  b_2(v) r^2  \leq z=f(r v) \leq (1+\d)  b_2(v) r^2,  $$
which implies
\begin{equation}\label{r}   (1+\d)^{-{\frac 1 2}}  b_2(v)^{-{\frac 1  2}} z^{\frac 1  2} \leq r=r(v,z)
\leq (1+\d)^{\frac 1  2}  b_2(v)^{-{\frac 1  2}} z^{\frac 1  2},  \end{equation}
where $r$ is the radial function of $K \cap H(z) $.
From this we obtain estimates for the $(d-1)$-dimensional volume of $K \cap H(z)$
\begin{equation}\label{schnitt}
(1+\d)^{-\frac {d-1}{ 2}} c_1   \kappa (u)^{-{\frac 1  2}} z^{\frac{d-1 }{ 2}} \leq
\displaystyle \l_{d-1} (K \cap H(z))
\leq (1+\d)^{\frac{d-1 }{ 2}} c_1   \kappa (u)^{-{\frac 1  2}} z^{\frac {d-1 }{ 2}}
\end{equation}
with a suitable constant $c_1 >0$, where $\kappa (u) = \prod k_i$ is the Gaussian curvature of $K$ at $p$. 
By definition
\begin{equation}\label{v-int}
  \l_+(z)= \int_0^{z} \l_{d-1} (K \cap H(t))  \dint t,
\end{equation}
which by (\ref{schnitt}) implies
\begin{equation} 
(1+\d)^{-\frac{d-1 }{ 2}} {\frac{2 }{ d+1}}   c_1   \kappa (u)^{-{\frac 1 2}} z^{\frac{d+1 }{ 2}} \leq
\l_+ (z)
\leq (1+\d)^{\frac{d-1 }{ 2}} {\frac{2 }{ d+1}}   c_1   \kappa (u)^{-{\frac 1 2}} z^{\frac{d+1 }{ 2}} .
\end{equation}
For given $z$, (\ref{r}) shows that $K \cap H(z)$ contains an ellipsoid ${\mathcal E}_-$ defined \linebreak by $(1+\d)^{-1}  b_2(v) r^2  =z$,
resp., is contained in an ellipsoid ${\mathcal E}_+$ defined by $(1+\d) b_2(v) r^2  =z$.
We are interested in 
$${\mathcal I}_{K \cap H(z)} = \int_{K \cap H(z)} \cdots \int_{K \cap H(z)}  \l_{d-1}( F)^2 \,  d x_1 \cdots \dint x_d 
.$$
Clearly, if the range of integration is increased, resp., decreased, ${\mathcal I}$ will increase, resp., decrease.
$${\mathcal I}_{{\mathcal E}_-} \leq {\mathcal I}_{K \cap H(z)} \leq {\mathcal I}_{{\mathcal E}_+} $$
Note that these integrals are invariant under volume--preserving affinities. Thus, ${\mathcal I}_{{\mathcal E}_\pm}$ does not depend on the shape of the ellipsoids and is proportional to $\l_{d-1} ({\mathcal E}_{\pm})^{d+2}$.
Hence, there exists a suitable constant  $c_2$ for which
\begin{eqnarray*}\label{e}   
(1+\d)^{-{\frac{(d-1)(d+2)}  2}} \lefteqn{ c_2   \kappa (u)^{-\frac{d+2 }{ 2}} z^{\frac{(d-1)(d+2) }{ 2}} }&&
\\ &\leq&
{\mathcal I}_{K \cap H(z)}
\leq (1+\d)^{\frac{(d-1)(d+2) } 2} c_2   \kappa (u)^{-\frac{d+2 }{ 2}} z^{\frac{(d-1)(d+2) }2} .  
\end{eqnarray*}
In the last step we estimate 
$$
{\mathcal J}_{K\cap H_+(z)} = \int_{K_+(z)} \dint _H(x) \, \dint x 
=
\int_0^z \l_{d-1}(K \cap H(t)) (z-t) \dint t.
$$
By the same monotonicity argument used above we obtain
\begin{equation}\label{e2}   
(1+\d)^{-\frac{d-1} 2} c_3  \kappa (u)^{-{\frac 1  2}} z^{\frac{d+3 } 2} \leq
{\mathcal J}_{K \cap H_+(z)}
\leq 
(1+\d)^{\frac{d-1} 2} c_3 \kappa (u)^{-{\frac 1 2}} z^{\frac{d+3 }2} 
.  \end{equation}

Now we are ready to estimate the integral
$$ \int_{h_K (u)-\e}^{h_K (u)}   e^{-t\l_+} {\mathcal I}_{K \cap H} {\mathcal J}_{K \cap H_+} \, \dint h   
= 
\int_{0}^{\e}   e^{-t\l_+(z)} {\mathcal I}_{K \cap H(z)} {\mathcal J}_{K \cap H_+(z)} \, \dint z.$$
Note that (\ref{v-int}) is equivalent to
$$ \frac{d }{ dz} \l_+(z)= -\l_{d-1} (K \cap H(z)) ,$$
and substituting $v=\l_+(z)$ implies
\begin{eqnarray*}
&&\int_0^{\e}  e^{-t \l_+(z)}  {\mathcal I}_{K \cap H(z)} {\mathcal J}_{K \cap H_+(z)}\, \dint z \\
&&\qquad =
\int_0^{\l_+(\e)} e^{-t v}  {\mathcal I}_{K \cap H(z(v))}  {\mathcal J}_{K \cap H_+(z(v))}   \l_{d-1}(K \cap H(z(v)))^{-1}\, \dint v ,
\end{eqnarray*}
where $H(z(v))$ denotes the hyperplane parallel to $z=0$ cutting off from $K$ a cap of volume $v$.

Combining this with (\ref{schnitt}) - (\ref{e2})  yields
\begin{eqnarray*}
c_4 \lefteqn{(1+\d)^{-\frac {(d-1)(d^2 +3d+3)}{(d+1)} }
 \kappa (u)^{-\frac d{d+1}}  
 \int_0^{\l_+(\e)} e^{-tv} v^{ \frac{d^2 +d+2}{d+1} }  \dint v } &&
\\ &\leq&
\int_{h_K (u)-\e}^{h_K (u)}  e^{-t v}  {\mathcal I}_{K \cap H(z)} {\mathcal J}_{K \cap H_+(z)}\ \dint h 
\\ & \leq  & 
c_4 (1+\d)^{\frac {(d-1)(d^2 +3d+3)}{(d+1)} }
\kappa (u)^{-\frac d{ d+1}}  
 \int_0^{\l_+(\e)} e^{-tv}
v^{\frac{d^2 +d+2}{d+1}}  \dint v
\end{eqnarray*}
with a suitable constant $c_4$.
Hence, we are interested in the asymptotic behavior of the Laplace transform 
$$ 
\int_0^{\l_+(\e)} e^{-tv}
v^{\frac{d^2 +d +2} {d+1}}  \dint v 
= {\mathcal L} \left( v^{\frac{d^2 +d +2} {d+1}} \right) (t) \\
+ O \left( (1-\gamma_2)^{t} \right)
$$
as $t \to \infty$. (Recall that $\l_+ (\e) \geq \gamma_2$.) By an Abelian theorem, cf., e.g., Doetsch \cite{Doe}, chap.~3, \S~1, we obtain
$$ 
{\mathcal L} \left \lbrace  v^{\alpha}  \right \rbrace (t)
=
\Gamma \left(\alpha +1 \right) t^{-\alpha-1}
+ O \left( t^{-\alpha-2}\right) \, $$
as $t \to \infty$. This implies the following bounds
\begin{eqnarray*}
c_5 \lefteqn{(1+\d)^{-\frac {(d-1)(d^2 +3d+3)}{(d+1)} }
 \kappa (u)^{-\frac d{ d+1}} 
t^{-(d+1)-\frac{2}{d+1}} (1+ O \left( t^{-1}\right)) }&&
\\ &\leq&
\int_{h_K (u)-\e}^{h_K (u)}  e^{-t v}  {\mathcal I}_{K \cap H(z)} {\mathcal J}_{K \cap H_+(z)}\ \dint h 
\\ & \leq  & 
c_5 (1+\d)^{\frac {(d-1)(d^2 +3d+3)}{(d+1)} }
 \kappa (u)^{-\frac d{ d+1}} 
t^{-(d+1)-\frac{2}{d+1}} (1+ O \left( t^{-1}\right)) 
\end{eqnarray*}
as $t \to \infty$, where the constants in $O(\cdot)$ and the constant $c_5$ are independent \linebreak of $p $ and $ u$.

Concerning the remaining integration note that the term
$$\int_{S^{d-1}}  \kappa(u)^{-1 + \frac 1 {d+1}} \dint u = \int_{\partial K}  \kappa(x)^{\frac 1 {d+1}} \dint x$$
is the affine surface area $\Omega( K)$. Since the terms in (\ref{gamma}) are of smaller order, we finally obtain
\begin{eqnarray*}
c_6 \lefteqn{(1+\d)^{-\frac {(d-1)(d^2 +3d+3)}{(d+1)} } \O(K)
t^{-(d+1)-\frac{2}{d+1}} (1+ O \left( t^{-1}\right)) }&&
\\ &\leq&
\int \limits_K \cdots \int \limits_K e^{-t\l_{d}(K_+(F))}  \lefteqn{ \l_{d-1}(F) d_{{\rm aff} F}(x)\,   dx_1 \cdots dx_d dx }
\\ & \leq  & 
c_6 (1+\d)^{\frac {(d-1)(d^2 +3d+3)}{(d+1)} } \O(K) 
t^{-(d+1)-\frac{2}{d+1}} (1+ O \left( t^{-1}\right)) 
\end{eqnarray*}
as $t \to \infty$ with a suitable constant $c_6$. Since this holds for each $\d >0$,  the proof is finished. \qed
\end{proof}

\noindent
\begin{footnotesize}
{\bf Acknowledgements.} 
M. Beermann was supported in part by the FWF project P 22388-N13, \lq Minkowski valuations and geometric inequalities\rq . We are grateful to an anonymous referee for careful reading of the manuscript and numerous helpful suggestions.
The final publication is available at Springer via  http://dx.doi.org/10.1007/s00454-014-9649-7.
\end{footnotesize}


\begin{thebibliography}{10}

\bibitem{Af1}
Affentranger, F.: {The expected volume of a random polytope in a ball}.
\newblock {J. Microscopy\/} {\bf 151}, 277--287  (1988)

\bibitem{Bar2}
B{\'a}r{\'a}ny, I.: {Random polytopes in smooth convex bodies}.
\newblock {Mathematika\/} {\bf 39}, 81--92  (1992)

\bibitem{BRe1}
B{\'a}r{\'a}ny, I., Reitzner, M.: Poisson Polytopes.
\newblock {Ann. Probab.\/} {\bf 38}, 1507--1531  (2010)

\bibitem{BRe2}
B{\'a}r{\'a}ny, I., Reitzner, M.: On the variance of random polytopes.
\newblock {Adv. Math.\/} {\bf 225}, 1986--2001  (2010)

\bibitem{Bu2}
Buchta, C.: {Stochastische {A}pproximation konvexer {P}olygone}.
\newblock {Z. Wahrsch. Verw. Geb.\/} {\bf 67}, 283--304 (1984)

\bibitem{Bu3}
Buchta, C.: {Zufallspolygone in konvexen {V}ielecken}.
\newblock {J. reine angew. Math.\/} {\bf 347}, 212--220 (1984)

\bibitem{Bu11}
Buchta, C.: {An identity relating moments of functionals of convex hulls}.
\newblock {Discrete Comput. Geom.\/} {\bf 33}, 125--142  (2005)

\bibitem{BM}
Buchta, C., M{\"u}ller, J.: {Random polytopes in a ball}.
\newblock {J. Appl. Probab.\/} {\bf 21}, 753--762 (1984)

\bibitem{BR2}
Buchta, C., Reitzner, M.: {Equiaffine inner parallel curves of a plane convex
  body and the convex hulls of randomly chosen points}.
\newblock {Probab. Theory Relat. Fields\/} {\bf 108}, 385--415  (1997)

\bibitem{BR4}
Buchta, C., Reitzner, M.: {The convex hull of random points in a tetrahedron:
  Solution of Blaschke's problem and more general results}.
\newblock {J. reine angew. Math.\/} {\bf 536}, 1--29 (2001)

\bibitem{CaYu}
Calka, P., Yukich., J.~E.: {Variance asymptotics for random polytopes in smooth
  convex bodies}.
\newblock {Probab. Theory Relat. Fields\/} {\bf 152}, 435--463 (2014)

\bibitem{Doe}
Doetsch, G.: {Handbuch der Laplace-Transformation II\/}.
\newblock Birkh{\"a}user, Basel Stuttgart (1955)

\bibitem{Ef}
Efron, B.: {The convex hull of a random set of points}.
\newblock {Biometrika\/} {\bf 52}, 331--343 (1965)

\bibitem{Hugsurv}
Hug, D.: {Random polytopes. In: Spodarev, E. (ed.) Stochastic geometry, spatial statistics and random fields}, \newblock {Lecture Notes in Math.\/} {\bf 2068}, pp. 205--238, Springer, Heidelberg (2013)

\bibitem{HMR}
Hug, D., Munsonius, G.~O., Reitzner, M.: {Asymptotic mean values of Gaussian
  polytopes}.
\newblock {Beitr. Algebra Geom.\/} {\bf 45}, 531--548 (2004)

\bibitem{Ki}
Kingman, J. F.~C.: {Random secants of a convex body}.
\newblock {J. Appl. Probab.\/} {\bf 6}, 660--672 (1969)

\bibitem{Re6}
Reitzner, M.: {Random polytopes and the Efron--Stein jackknife inequality}.
\newblock {Ann. Probab.\/} {\bf 31}, 2136--2166 (2003)

\bibitem{Reitsurv}
Reitzner, M.: {{Random polytopes. In: Kendall, W., Molchanov, I. (eds.)  New perspectives in stochastic geometry}}, pp. 45--76,  Oxford Univ. Press, Oxford (2010)

\bibitem{RS1}
R{\'e}nyi, A., Sulanke, R.: {{\"U}ber die konvexe H{\"u}lle von $n$
  zuf{\"a}llig gew{\"a}hlten Punkten}.
\newblock {Z. Wahrsch. Verw. Geb.\/} {\bf 2}, 75--84 (1963)

\bibitem{RS2}
R{\'e}nyi, A., Sulanke, R.: {{\"U}ber die konvexe H{\"u}lle von $n$
  zuf{\"a}llig gew{\"a}hlten Punkten II}.
\newblock {Z. Wahrsch. Verw. Geb.\/} {\bf 3}, 138--147 (1964)

\bibitem{SchnWe3}
Schneider, R., Weil, W.: {Stochastic and integral geometry\/}.
\newblock Probability and its Applications (New York). Springer-Verlag, Berlin
  (2008)

\bibitem{Zi}
Zinani, A.: {The expected volume of a tetrahedron whose vertices are chosen at random in the interior of a cube}.
\newblock {Monatsh. Math.\/} {\bf 139}, 341--348 (2003)

\end{thebibliography}
\end{document}